\documentclass[11pt]{amsart}
\pdfoutput=1
\usepackage{setspace}
\usepackage{amsmath,amsfonts,amsthm,graphicx, verbatim}
\usepackage{hyperref}
\usepackage{graphicx,psfrag}  
\usepackage[psamsfonts]{amssymb}
\usepackage{amscd}
\usepackage[all,cmtip]{xy}
\usepackage{mathtools}

\title{Symmetries of geodesic flows on covers and rigidity}
\author{Daniel Mitsutani}
\address[Mitsutani]{Department of Mathematics, the University of Chicago, Chicago, IL, USA, 60637}
\email{mitsutani@math.uchicago.edu}

\date{\today}
\theoremstyle{plain}
\newtheorem{main}{Theorem}

\newtheorem{theorem}{Theorem}[section]
\newtheorem{proposition}[theorem]{Proposition}
\newtheorem{lemma}[theorem]{Lemma}
\newtheorem{corollary}[theorem]{Corollary}

\newtheorem{claim}[theorem]{Claim}
\newtheorem{quest}[theorem]{Question}
\newtheorem{definition}[theorem]{Definition}
\theoremstyle{remark}
\newtheorem{remark}[theorem]{Remark}
\newtheorem{remarks}[theorem]{Remarks}

\def\eps{\varepsilon}

\def\i{\iota}

\def\Diff{\operatorname{Diff} }

\def\Isom{\hbox{Isom} }

\def\dim{\hbox{dim} \, }

\def\title{\em}

\def\g{\frak g}
\def\bls{\backslash}
\def\h{\frak h}

\def\cW{\mathcal{W}}

\def\cV{\mathcal{V}}
\def\cF{\mathcal{F}}
\def\cE{\mathcal{E}}

\def\Conea{C^{1,\alpha}}

\def\cU{\mathcal{U}}
\def\cG{\mathcal{G}}
\def\cA{\mathcal{A}}

\def\lieg{\mathfrak{g}}

\def\cD{\mathcal{D}}

\def\wW{\wt{W}}

\def\quart{\frac{1}{4}}
\def\tM{\widetilde{M}}

\def\cH{\mathcal{H}}

\def\cP{\mathcal{P}}

\def\wt{\widetilde}
\def\ol{\overline}

\def\transverse{\,\raise2pt\hbox to1em{\hfil$\top$\hfil}\hskip -1em \hbox
to1em{\hfil$\cap$\hfil}\,} 
\newcommand\Z{\mathbb{Z}}
\newcommand\R{\mathbb R}
\newcommand\C{\mathbb{C}}

\newcommand\N{{\mathbb N}}

\newlength{\figboxwidth} 

\begin{document}

\begin{abstract} We define and study the \textit{foliated centralizer}: the group of $C^\infty$ centralizer elements of the lift of an Anosov system on a non-compact manifold which additionally preserve the stable and unstable foliations. When the Anosov system is the geodesic flow of a closed Riemannian manifold with pinched negative sectional curvatures, we prove some rigidity properties for the foliated centralizer of the lift of the flow to the universal cover: it is a finite-dimensional Lie group, which is moreover discrete (modulo the action of the flow itself) unless the metric is homothetic to some real hyperbolic metric. This result is inspired by the study of isometries of universal covers that appeared originally in the work of Eberlein \cite{eb2}, and later in Farb and Weinberger in \cite{fawein}, as well as centralizer rigidity results in dynamics.
\end{abstract}

\maketitle

\section{Introduction}

It is a widely explored theme in geometry that symmetries occur mostly in special circumstances. For instance, a generic Riemannian metric $g$ on a closed manifold $M$ admits no isometries. For metrics of negative Ricci curvature an old result of Bochner shows that  closed manifolds admit only finitely many isometries.  While it is easy to perturb a locally symmetric metric of negative sectional curvature on a closed manifold without changing the isometry group, in this setting it has been known that rigidity occurs when such a metric admits ``too many" isometries in some sense. The idea, pioneered by Eberlein, is to look for excess symmetries on covers: if the isometry group of the universal cover of an $(M,g)$ of negative sectional curvature is not discrete then $g$ is locally symmetric \cite{eb2}. This result was later extended with other methods in great generality in \cite{fawein} to any aspherical Riemannian manifold $M$ with an arbitrary metric.

In this paper we will assume that $M$ is a closed manifold and $g$ is a smooth Riemannian metric of negative sectional curvature $K \leq -1$. Associated to $g$ is a hyperbolic dynamical system on the unit tangent bundle $SM$ referred to as the geodesic flow $\varphi_t$, which flows a unit tangent vector along its tangent unit-speed oriented geodesic for time $t$. In our setting, the geodesic flow $\varphi_t$ is an \textit{Anosov flow}, meaning that there exists a $D\varphi^t$-invariant splitting $TSM = E^u \oplus E^s \oplus \R X_g$ and constants $\tau > 1, C>1$ so that for $t \geq 0$:
$$\begin{aligned}\label{eq: contractrate}
		& C^{-1}e^t \leq m(D\varphi^t|_{E^u}) \leq \|D\varphi^t|_{E^u}\| \leq Ce^{\tau t}, \hspace{.2cm}\\  & C^{-1}e^t \leq m(D\varphi^{-t}|_{E^s}) \leq \|D\varphi^{-t}|_{E^s}\| \leq Ce^{\tau t},
	\end{aligned}
$$
and where $X_g$ generates the flow $\varphi_t$. The bundles $E^u$ and $E^s$ are known to integrate to \textit{unstable and stable foliations} $W^u$ and $W^s$ of $SM$ which project to the \textit{horospherical foliations} of $M$, that is, the level sets of the Busemann functions on $M$ as a Hadamard space (see Subsection \ref{ssec: geodesicflows}) .

For $r \geq 0$ an integer, the symmetries of a $C^r$-flow $\varphi^t: X \to X$  are described by the centralizer group: the group $Z^r(\varphi^t)$ of $C^r$-diffeomorphisms of $X$ commuting with the system, i.e.:
$$Z^r(\varphi^t) = \{f \in \Diff^r(X): f \circ \varphi^t = \varphi^t \circ f, \, \forall t \in \R\}.$$

The elements of $Z(\varphi^t)$ are known to be rare in a generic sense as well \cite{bcw}. Moreover, when the dynamical system has some hyperbolicity such as the Anosov condition, in many settings it has been shown that a large centralizer group implies rigidity, generally meaning that the dynamical system is conjugate to an algebraic one. \textit{Our main result (Theorem \ref{main}) shows that, in the spirit of Eberlein's result \cite{eb2}, this occurs for the geodesic flow of a closed Riemannian manifold of negative sectional curvature, provided that as explained in the first paragraph one looks at the universal cover to find excess symmetries.}

First however, we have to determine more carefully the notion of a symmetry of the geodesic flow on the universal cover $S\wt{M}$. The centralizer group of the geodesic flow on the universal cover, which would be a natural candidate, turns out to always be too large, i.e. always infinite-dimensional (see Proposition \ref{notcomp}) since it is possible to construct perturbations which are localized around a wandering orbit.  In a non-compact setting, some uniformity of the centralizer elements is desirable since the notion of hyperbolicity itself depends on the metric structure chosen. Moreover, a uniformly continuous diffeomorphism of $S\wt{M}$ must preserve the stable and unstable foliations of the geodesic flow. 

In  fact, this latter condition -- that the centralizer element also preserves the  stable and unstable foliation -- will suffice for our theorem, and turns out to be a natural requirement in the classification up to conjugacy of Anosov systems in non-compact settings. For instance, recently in the study of Anosov diffeomorphisms in non-compact settings Groisman and Nitecki \cite{gn} have introduced the notion of a \textit{foliated conjugacy} -- a conjugacy between two diffeomorphisms preserving stable and unstable foliations -- and constructed many Anosov diffeomorphisms on $\R^2$  which can be distinguished by foliated conjugacy classes.

This leads us to the following definition for the group of symmetries of the geodesic flow and the stablee and unstable foliations simultaneously. Again, recall that $(M,g)$ is a negatively curved closed Riemannian manifold and let $S\wt{M}$ be the unit tangent bundle of the universal cover of $M$.
\begin{definition}  The \textit{foliated centralizer} of the geodesic flow on $S\wt{M}$ is the group $G < \Diff^\infty(S\wt{M})$ of all $\psi \in \Diff^\infty(S\wt{M})$ satisfying:
	\begin{enumerate}
		\item [(a)] $d\psi(X_g) = X_g$, where $X_g$ is the geodesic spray, 
		\item [(b)] $\psi$ preserves the horospherical foliations of $S\wt{M}$, i.e., $\psi(W^{u,s}(v)) = W^{u,s}(\psi (v))$.
	\end{enumerate}
\end{definition}

As is the case with isometries, we expect the presence of too many symmetries for the geodesic flow on the universal cover to determine a rigid situation. This is the main result we prove. In what follows, the foliated centralizer $G$ as defined above is endowed with the compact-open topology. Let $G_c = G/\langle\wt{\varphi}^t \rangle $ be the \textit{reduced foliated centralizer}, obtained by taking the quotient of $G$ by the geodesic flow on the universal cover $\wt{\varphi}^t$. We will say that $G$ is trivial if $G_c$ is discrete. A priori $G_c$ is only a Hausdorff topological space, but as we will see later (Theorem \ref{lie}) under our assumptions $G_c$ is in fact always a Lie group.

\begin{main}\label{main}  For $n \geq 3$, let $(M^n,g)$ be a closed Riemmanian manifold with negative strictly $\frac{1}{4}$-pinched sectional curvatures, that is, with sectional curvatures $-4 < K \leq -1$. If the reduced foliated centralizer $G_c$ is non-trivial, then $(M,g)$ is homothetic to a real hyperbolic manifold.
\end{main}
\begin{remark}
It is well-known that the pinching assumption can be taken to be pointwise rather than global for the results we introduce in Section \ref{sec: prel} to hold. That is, one may instead assume that  for every $x \in M$ the ratio of the sectional curvatures at $x$ is $\quart$-pinched. Using this and replacing the use of \cite{bfl} at the end of the proof of the theorem by the analogous result in \cite{hurderkatok} it is easy to see that the result also holds for negatively curved surfaces even without the pinching assumption.
\end{remark}

The $\quart$-pinching hypothesis is necessary for the theorem to be true -- indeed, any closed, complex hyperbolic manifold has non-trivial foliated centralizer. This however opens up the question of whether the theorem can be true without the $\quart$-pinching hypothesis but replacing the conclusion with $(M,g)$ is locally symmetric. A major difficulty in proving this generalization is that the Anosov splitting is in general not $C^1$ if the sectional curvatures are not $\quart$-pinched, as we discuss in Section \ref{sec: rigidgeometricstructures}. 

\subsection{Hidden Symmetries.} Farb and Weinberger \cite{hidsym} proved that on arithmetic manifolds only locally symmetric Riemannian metrics admit infinitely many isometries on finite covers which are not lifts of other isometries. Since these isometries can only be detected on covers, they call them \textit{hiddden symmetries}.  For manifolds of negative sectional curvature, this result is immediately recovered from Eberlein's result, since $[\Isom(\wt{M}): \pi_1(M)] = \infty$  implies that $  \Isom(\wt{M})$ is not discrete for $M$ compact. 

Theorem \ref{main} almost allows us to recover an analogous result for the centralizer groups of finite covers of $SM$. Recall that $\psi \in \Diff^\infty(SM)$ is a centralizer element of the geodesic flow $\varphi_t$ of a closed Riemannian manifold if $\psi \circ \varphi_t = \varphi_t \circ \psi$ for all $t \in \R$ or equivalently if $d\psi(X_g) = X_g$, where $X_g$ is the geodesic spray. Centralizer elements on closed Riemannian manifolds of negative sectional curvature preserve the horospherical foliations, so their lifts to the universal cover lie in the foliated centralizer. It is thus natural to ask:

\begin{quest}\label{hidsym} For $n \geq 3$, let $(M^n,g)$ be a closed Riemmanian manifold with strictly $\frac{1}{4}$-pinched  negative sectional curvature. Suppose there exist infinitely many (up to the action of the flow itself) centralizer elements of the geodesic flow of finite Riemannian covers of $(M,g)$ which are not lifts of centralizer elements of other covers of $SM$. Then is $(M,g)$  homothetic to a real hyperbolic manifold?
\end{quest}

Unlike the isometry case, where $[\Isom(\wt{M}): \pi_1(M)] = \infty$ immediately implies that $  \Isom(\wt{M})$ is not discrete since $\Isom(\wt{M})/ \pi_1(M)$ is compact, for the centralizer it is not obvious that the existence of infinitely many centralizer elements on finite covers implies that the foliated centralizer is not trivial.

The differential of an isometry of a Riemannian manifold is a centralizer element of the associated geodesic flow, so this result would evidently extend the hidden symmetry theorem for negatively curved $\quart$-pinched metrics to ``hidden symmetries"  of the geodesic flow. An interesting open question, in the negative curvature setting, is whether in fact every centralizer element of the geodesic flow is the differential of an isometry (up to composition with the flow itself). In this direction, a positive answer to Question  \ref{hidsym} would also show that there can only be finitely many more centralizer elements on all finite covers of $M$ than there are isometries. 

\subsection{Rigid Geometric Structures.} \label{sec: rigidgeometricstructures} Lastly, we point out that the foliated centralizer has in fact  been previously studied with rigidity applications in mind as the group of isometries of the structure determined by the geodesic spray $X_g$ and the bundles $E^u$ and $E^s$ tangent to the horospherical foliations. It has long been known that $(X_g, E^u, E^s)$ determines a pseudo-Riemannian metric $h_{KC}$ on $SM$, referred to as the Cartan-Kanai form (see Subsection \ref{sssec: kanai}). Since $h_{KC}$ is determined by $(X_g, E^u, E^s)$, the foliated centraliizer $G$ is a subgroup of the (global) isometry group of $h_{KC}$.  

When the horospherical foliations are smooth, $h_{KC}$  is a \textit{rigid geometric structure} in the sense of Gromov. In its most general formulation, a rigid geometric structure   \cite{gr} $T$ on a closed Riemannian manifold $V$ in the sense of Gromov is a geometric structure, such as a tensor, whose pseudo-groups of local isometries are finite-dimensional. Typical examples of such $T$ are affine connections and smooth Riemannian or pseudo-Riemannian metrics.  In a well-known paper, Benoist, Foulon and Labourie \cite{bfl} use this fact to show, using Gromov's theory, that a closed Riemannian manifold with negative sectional curvature and smooth horospherical foliations must be locally symmetric.

When $M$ is a closed manifold of negative curvature with $C^r$ horospherical foliations, the pseudo-Riemannian metric $h_{KC}$ is still well-defined but only has $C^r$ regularity.  However, as is well known, closed Riemannian manifolds of negative sectional curvature admit only Hölder continuous horospherical foliations in general, so $h_{KC}$ does not define a rigid geometric structure in the sense of Gromov, since in low regularity pseudo-Riemannian metrics can admit infinite dimensional local isometry pseudo-groups. On the other hand, while requiring $C^2$ regularity of the horospherical foliations would give that $G$ is a (finite-dimensional) Lie group, this proves too restrictive since conjecturally this situation can only happen when $M$ is locally symmetric. 

A natural class of intermediary $C^1$ regularity of horospherical foliations occurs when the sectional curvature of $M$ are $\quart$-pinched. In this situation $h_{KC}$ is still only $C^1$, so it is not immediately obvious that the foliated centralizer is finite-dimensional. Thus, the natural first step of the proof of Theorem \ref{main} is to show that $G$ is a Lie group:

\begin{main}
	\label{lie} Let $(M,g)$ be a closed Riemannian manifold with $C^1$ horospherical foliations. Then the foliated centralizer $G$ is a finite dimensional Lie group.
\end{main}

The proof of Theorem \ref{lie} uses the fact that elements of $G$ define affine transformations with respect to the \textit{Kanai connection} $\nabla$, an affine connection associated to $h_{KC}$, to show that $G$ is locally compact (Section \ref{sec: compact}).  Then by the classical work of Montgomery and Zippin on Hilbert's fifth problem this implies that $G$ is a finite-dimensional Lie group.

\begin{remark}
	
	The use of the classical theory of Montgomery-Zippin to show certain transformation groups in dynamics are finite dimensional Lie groups has recently yielded many rigidity results for centralizers; for instance, simultaneously to the writing of this work, Damjanovic-Wilkinson-Xu \cite{dwxpre} have also announced rigidity results which use this to classify the centralizer of certain partially hyperbolic diffeomorphisms. 
\end{remark}

\subsection{Outline of the Proof.}

After we conclude that $G$ is a Lie group (Section \ref{sec: compact}) we assume as in the hypothesis of Theorem \ref{main} that $G_c$ is not discrete. Then non-discreteness of $G_c$ implies by classic results also of \cite{mz} that it must define a smooth 1-parameter action on $\wt{V}$. We show that the partition of $\wt{V}$ by the orbits of this action must contain an open set of full dimension by an argument due to Hamenstadt using 1/4-pinching, so that the action of $G$ admits an open and dense orbit in which the horospherical foliations must be smooth (Section \ref{sec:smooth}). Finally, using the dynamics of the action of $\pi_1(M)$ on $\partial \wt{M}$, we show that the horospherical foliations must in fact be smooth everywhere. The result then follows from the main theorem of \cite{bfl}.

\subsection{Acknowledgements} I would like to thank Amie Wilkinson for all her suggestions and continued guidance in the process of research leading up to this paper, and also for her help in reviewing the text. I also thank Boris Hasselblatt and Ursula Hamenstadt for their useful replies to my inquiries and Jonathan DeWitt for his help reviewing the text.

\section{Preliminaries} \label{sec: prel}

In this section, we review important constructions and give precise definitions needed for the proofs to follow.

\subsection{Topological Transformation Groups} We review some concepts related to transformation groups and Montgomery-Zippin's theory that are fundamental to our results.

\subsubsection{$C^r$-topologies} \label{ssec: tgroups} Let $0 \leq r \leq \infty$ be an integer. For non-compact manifolds, there are two available choices of topologies on function spaces. For our purposes, we focus on the \textit{weak $C^r$-topology} on $C^r(M, N)$ for $M$ and $N$ smooth manifolds is defined as follows. For $f \in C^r(M,N)$, $(\varphi, U)$ and $(\psi, V)$ charts on $M$ and $N$ respectively, $K\subseteq M$ a compact set such that $f(K) \subseteq V$, $\eps > 0$ let $N^r(f; (\varphi, U), (\psi, V), K, \eps)$ be given by  $$\{g \in C^k(M,N): g(K) \subseteq V,  \|D^k(\psi f \varphi^{-1})(x) - D^k(\psi fg\varphi^{-1})(x)\| < \eps\},$$
where $D^k$ represents the $k$-th partial derivatives of order and $k$ runs over $k = 0, ..., r$ (when $r = \infty$, $k$ runs over $\N$). The weak $C^r$-topology, is defined to have the sets $N^r(f; (\varphi, U), (\psi, V), K, \eps)$ as a basis, which we from now on we simply refer to  as the $C^r$-topologies. In other words, it is the topology of uniform $C^r$-convergence on compact sets. The weak $C^r$-topologies are second countable \cite[p. 35]{hir}.

\subsubsection{Montgomery-Zippin's solution to Hilbert's Fifth Problem} \label{sssec:mz}

We recall some facts about group actions on manifolds. Throughout, let $M$ be a smooth manifold. A \textit{topological (resp. Lie) group of $C^r$-transformations} is a Hausdorff topological (resp. Lie) group $G$ with a homomorphism into $\Diff^r(M)$. Partially giving an answer to Hilbert's fifth problem, Gleason, Montgomery-Zippin and Yamabe proved a criterion for when topological groups admit a compatible (finite-dimensional) smooth manifold structure, i.e., are Lie groups, in terms of the no small subgroups (NSS) criterion. 

We will use the following consequence of their theorem, together with a result of Bochner-Montgomery \cite[Thm. 1]{mont} showing that a locally compact subgroup of $\Diff^r(M)$ has the no small subgroups property. Recall that we say the action of $G$ is effective if its kernel into $\Diff^r(M)$ is trivial.

\begin{theorem} \label{mz}
	\cite{mz, yam}
	Let $G$ be a topological group of $C^r$-transformations on a smooth
	manifold $M$, with $r \geq 1$. If $G$ is effective and locally compact, then $G$ admits a unique smooth structure which makes it is a Lie group of $C^r$-transformations.
\end{theorem}

Moreover, it follows from \cite[Thm. 4]{mont} that in the situation above the natural action map $G \times M \to M$ is of class $C^r$. In particular, as we will need later, the action of 1-parameter subgroups of $G$ corresponds to a $C^r$ action of $\R$ by $C^r$-diffeomorphisms on $M$.

\subsection{$C^r$-foliations} We lay out standard definitions and results relating to smooth foliations, using the same notations as \cite{clark}. 

For a smooth $m$-dimensional $C^s$-manifold $M$, $1 \leq k \leq m-1$ and $0 \leq r \leq s$,  a \textit{$k$-dimensional $C^r$-foliation} is a decomposition of $M$ into disjoint subsets $\{\cW_i\}_{i\in I}$ such that each $\cW_i$ is a connected $C^s$-submanifold of $M$ and for each $x \in M$ there is an open neighborhood $U$ of $x$ and a $C^r$-coordinate chart $\psi: U \to B_k \times B_{k-m} \subseteq \R^k \times \R^{k-m}$ such that $\psi(x) = (0,0)$ and for each $i \in I$ and each connected component $W$ of $U \cap \cW_i$ there is a unique $p \in B_{k-m}$ such that $\psi(W) = B_k \times \{p\}$.

\subsubsection{Foliations with continuously $C^r$-leaves} A foliation $\cW$ of a $C^s$-manifold $M$ is said to have \textit{continuously $C^r$ leaves} if there is an atlas of foliation charts $(\psi_j, U_j)_{j \in J}$ for $\cW$ with $\psi_j: U_j \to B_k \times B_{m-k}$ such that the compositions for $p \in B_{m-k}$, 
$$\zeta_{j,p} = \psi_j^{-1} \circ i_p: B_k \to M, $$
where $i_p(x) = (x,p)$ are $C^r$ and vary continuously on $p$ in $C^r(B_k, M)$.

For another smooth manifold $N$, $f: M \to N$ is \textit{continuously $C^r$ along $\cW$} if for each chart $(\psi_j, U_j)$ the compositions $f \circ \zeta_{j,p}$ depend continuously on $p \in B_{m-k}$ inside of $C^r(B_k, N)$.

\begin{remark}
	The continuously $C^r$ along $\cW$ condition is commonly referred to in the  literature as uniformly $C^r$ along $\cW$, especially in compact settings. We use the term continuously to emphasize that the leaves and maps do not have to vary uniformly continuously in $C^r$-topology in the non-compact setting.
\end{remark}

\subsubsection{Connections along foliations} \label{ssec: conn} Let $r \geq 0$ be an integer. Throughout, a $C^r$ connection on a $C^{r+1}$ vector bundle $E$ over $M$ is a linear map $\nabla: \Gamma^{r+1}(E) \to \Gamma^r(E \otimes T^*M)$, where $\Gamma^{r+1}(E)$ is the space of $C^{r+1}$-sections of $E$ .

Fix $M$ and $\cW$ a foliation with continuously $C^{r+1}$ leaves. Let $E \subseteq TM $ be a vector bundle over $M$ which is continuously $C^{r+1}$ along leaves. Let $\Gamma_\cW^{r+1}(E)$ be the space of sections $Z: M \to E$ which are continuously $C^{r+1}$ along $\cW$. Given $X \in \Gamma^{r+1}_\cW(E)$, we may consider it as a $C^{r+1} $ section $X: M_\cW \to E$, where $M_\cW = \sqcup_{i \in I} \cW_i$ (i.e., the topologically disjoint union of the leaves). Then a \textit{ $C^r$ connection on $E$ along $\cW$} is a $C^r$ connection on $E$ considered as a $C^{r+1}$ vector bundle over $M_{\cW}$. It is said to be \textit{continuously $C^r$ along $\cW$} if
$$ \nabla(\Gamma_\cW^{r+1}(E)) \subseteq \Gamma_\cW^r(T^* \cW \otimes E),$$
where we regard $\nabla$ as a map $\Gamma_\cW^{r+1}(E) \to \Gamma(T^* \cW \otimes E)$ by $X \mapsto (Z \mapsto \nabla_X Z)$.

\subsubsection{Journé's Theorem} Finally, in dealing with maps which are regular along a pair of transverse foliations of complementary dimensions, a basic tool is Journé's Theorem:

\begin{theorem} \cite{journe} \label{journe} Let $\cF^s$ and $\cF^u$ be two  transverse foliations of a manifold $M$ with continuously $C^\infty$ leaves. If $f\colon M \to N$ is continuously $C^\infty$ along the leaves of $\cF^s$ and $\cF^u$ then $f \in C^\infty(M,N)$. 
\end{theorem}

\subsection{Geodesic Flows in Negative Curvature} \label{ssec: geodesicflows} From now on, we fix a closed Riemannian manifold $(M,g)$ with negative sectional curvatures $K \leq  -1$ (always obtainable by homothety) and $\wt{M}$ its universal cover. We write $\Gamma := \pi_1(M)$, and observe that in what follows all forms are invariant by the action of $\Gamma$ by deck transformations on $\wt{M}$ and its associated tangent bundle, unless explicitly noted. Throughout, $V:= SM$ denotes the unit tangent bundle of a manifold in general, $\wt{V}$ is its universal cover, and the \textit{flip map} is denoted by $\i$, i.e. the diffeomorphism of $SM$ given by $(x,v) \mapsto (x,-v)$.

Let $\cV \subseteq TT\wt{M}$ be the \textit{vertical bundle} tangent to vector fields tangent to the fibers of $T\wt{M}$ and $\cH \subseteq TT\wt{M}$ be the \textit{horizontal bundle} which is tangent to parallel vector fields on $TM$. The Sasaki metric $g_{Sas}$ on $T\wt{M}$ is defined by pulling  back the metric $g \oplus g$ using the isomorphism $TT\wt{M} = \cV \oplus \cH \cong T\wt{M}  \oplus T\wt{M}$.  We denote the Sasaki norm simply by $\| \cdot \|$ and we endow $T\wt{M}$ and the unit tangent bundle $S \wt{M}$ with the Sasaki metric and with the metric space structure coming from the Riemannian distance $d_{Sas}$ given by the Sasaki metric on $T\wt{M}$ and $S\wt{M}$ respectively. When $g$ is complete, $g_{Sas}$ is complete and thus so is $S \wt{M}$ as a metric space. 

The ideal boundary $\partial M$ is defined as  equivalence classes $[\gamma]$ of oriented geodesics $\gamma$ which are asymptotic to each other as $t \to \infty$, and $\partial M$ is endowed with the quotient topology of the equivalence relation. The natural projection map $\wt{V} \to \partial M$ is given by $v \mapsto [\gamma_v]$, where $\gamma_v$ is the oriented geodesic through $v$. The set of preimages of a fixed $[\gamma_v]$ under this projection is denoted $\wt{W}^{cs}(v) \subseteq \wt{V}$, and as $v$ varies $\{\wt{W}^{cs}(v)\}_{v \in \wt{V}}$ defines a foliation of $\wt{V}$ named the center-stable foliation. Moreover, we define the stable foliation  $\{\wt{W}^{s}(v)\}_{v \in \wt{V}}$ by defining the sets $\wt{W}^{s}(v)$ to be the set of $w \in \wt{W}^{cs}(v)$ such that $d_{Sas} (\varphi^t(v), \varphi^t(w)) \to 0$ as $t \to  \infty$. The center-unstable $\{\wt{W}^{cu}(v)\}_{v \in \wt{V}}$ and unstable foliations $\{\wt{W}^{u}(v)\}_{v \in \wt{V}}$  are defined as the images of the center-unstable and unstable foliations respectively. Throughout, for $v \in SM$, we denote the projection of $v$ to $\partial M$ under $v \mapsto [\gamma_v]$ by $v_+ = [\gamma_v]$ and similarly we write $v_- = [\gamma_{\iota(v)}] \in \partial M$.

Lastly, given $v \in S\wt{M}$ with basepoint $p$, we define the horosphere $B_{v_+}(p) := \pi(\wt{W}^s(v)) \subseteq \wt{M}$, where $\pi: S\wt{M} \to \wt{M}$ is the standard projection map.
\subsubsection{Geodesic Flows}  \label{ssec: geodesic} 

For $(M, g)$ as above, let $\lambda^*$ be the canonical $1$-form on the cotangent bundle of the universal cover $T^*\tM$ and $\omega^* = d\lambda^*$. Denote by $X_g$ the geodesic spray, i.e., the vector field generating the geodesic flow $\varphi^t$ on $T\wt{M}$. Pulling back these forms by the bundle isomorphism $T^*\wt{M} \cong T \wt{M}$ given by the Riemannian metric $g$, we obtain forms $\lambda$ and $\omega$ on $T\tM$. Moreover, the isomorphism $TT\wt{M} \cong T^*T\wt{M}$ induced by the Sasaki metric sends $X_g$ to the $1$-form $\lambda$, which is to say that $\lambda(Y) = \langle X_g, Y \rangle_{Sas}$, providing an alternate description of the geodesic spray $X_g$. Moreover, the $1$-form $\lambda$ is invariant by the geodesic flow, i.e., $\mathcal{L}_{X_g}\lambda = 0$ or equivalently $\varphi^t_* \lambda = \lambda$ for all $t \in \R$. 

Recall that for a linear map $A$ between normed linear spaces its co-norm $m(A)$ is given by:  $$m(A) := \inf_{\|v\| = 1} \|Av\|.$$ The geodesic flow generated by $X_g$ is an \textit{Anosov flow} when restricted to the unit tangent bundle $V:= SM$, that is, there exists a $D\varphi^t$ (we denote the geodesic flow on $SM$ also by $\varphi^t$) splitting $TV = E^u \oplus E^s \oplus \R X_g$ and constants $\tau > 1, C>1$ so that for $t \geq 0$:
\begin{equation}\begin{aligned}\label{eq: contractrate}
& C^{-1}e^t \leq m(D\varphi^t|_{E^u}) \leq \|D\varphi^t|_{E^u}\| \leq Ce^{\tau t}, \hspace{.2cm}\\  & C^{-1}e^t \leq m(D\varphi^{-t}|_{E^s}) \leq \|D\varphi^{-t}|_{E^s}\| \leq Ce^{\tau t}.
\end{aligned}
\end{equation}

By the standard theory of hyperbolic dynamical systems, it is known that the unstable and stable bundles $E^u, E^s$ are integrable with tangent foliations $W^u$ and $W^s$ with continuously $C^\infty$ leaves. The leaves $W^u$ and $W^s$ when lifted to the universal cover agree with the unstable and stable foliations described in the last section, hence their common notation. The center unstable and stable bundles $E^{cu,cs}:= E^{u,s} \oplus \R X_g$ have associated  foliations $W^{cu}$ and $W^{cs}$ with continuously $C^\infty$ leaves with the same properties.  

Finally, we recall that the lifted foliations $\wW^\bullet$ for a negatively curved space have \textit{global product structure}, namely, for any two leaves of $W^{cu, cs}(x)$ and $W^{s,u}(y)$, there exists a unique intersection point $W^{cu,cs}(x) \cap W^{s,u}(y)$ as long as $y \neq \i(x)$. 

\subsubsection{Regularity of the splitting, fiber-bunching, holonomies}  \label{ssec:hols} For an Anosov flow the bundles $E^{cu,cs}$ are in fact Hölder-continuous with Hölder constant determined by the constant $\tau$ in the inequalities (\ref{eq: contractrate}) . Moreover, it is well known that for geodesic flows, the constant $\tau$ is directly related to the pinching constant of sectional curvature. Precisely, it is a standard fact regarding the solutions to the Ricatti equation that when the pointwise sectional curvatures of $M$ are pinched in the interval $[-a, -1]$, then in fact $\tau = \sqrt{a}$.

Based on the considerations above, Hasselblatt proved a sharp result on the regularity of the stable and unstable foliations:

\begin{theorem}\cite{hass} \label{regularity} Let $(M,g)$ have sectional curvature pinched in the interval $[-a, -1]$. Then the splitting $TSM = E^u \oplus E^s \oplus \R X_g$ is of class $C^{2/\sqrt{a}-}$, where regularity $C^{r-}$ for $r > 0$ means regularity $C^{\lfloor r \rfloor, r - \lfloor r \rfloor - \eps}$ for all $\eps > 0$.
\end{theorem}

The next related definition is of a fiber bunched system. Recall that a cocycle over a flow $\varphi^t:X \to X$ on a smooth manifold $X$ is a map  $\mathcal{A}: \mathcal{E} \times \R \to \mathcal{E}$ where $\cE$ is a fiber bundle over $X$ such that if $\pi: \cE \to X$ is the canonical projection we have $\pi \circ \cA^t = \varphi^t \circ \pi$ where $\cA^t := \cA(\cdot, t)$. 

\begin{definition}[Fiber Bunching] \label{fbunchdef} Let $X$ a smooth closed manifold and let $\cE$ be a $\beta$-Hölder continuous subbundle of $TX$. A $\beta$-Hölder continuous cocycle $\mathcal{A}: \mathcal{E} \times \R \to \mathcal{E}$ over an Anosov flow $\varphi^t:X \to X$ is said to be $\alpha$-fiber bunched if $\alpha \leq \beta$ and there exists $T> 0$ such that for all $p \in M$ and $t \geq T$:
	$$\|A^t_p\| \|A_p^{-t}\| \| D\varphi^t|_{E^s}\|^{\alpha} < 1, \hspace{.3cm} \|A^t_p\| \|A_p^{-t}\| \| D\varphi^{-t}|_{E^u}\|^{\alpha} < 1.$$
	
	When the inequalities are satisfied by $\cA = D\varphi^t|_{E^{u}}$ and $\cA = D\varphi^t|_{E^{s}}$ the Anosov flow itself is said to be $\alpha$-bunched.
\end{definition}

Besides regularity of the Anosov splitting, the other main reason we restrict ourselves to fiber-bunched systems is the existence of \textit{holonomies}. For a fiber bundle as in Definition \ref{fbunchdef} and for $x, y \in X$ sufficiently close let $I_{xy}: \cE_x \to \cE_y$ be the identification obtained by parallel transporting (with respect to any fixed smooth Riemannian metric on $X$) $\cE_x \subseteq T_xX$ to $T_yX$ along the unique shortest geodesic from $x$ to $y$ and orthogonally projecting so that dependence of $I_{xy}$ on $x,y$ is also $\alpha$-Hölder continuous, and let $I_x:\cE_x \to \cE_x$ be the identity map. The following proposition is standard in this setting, see e.g. \cite[Proposition 2.3]{clark} :

\begin{proposition}[Cocycle holonomies] \label{holdef} Let $\cA$ be a cocycle as in Definition \ref{fbunchdef}. Then for each $x \in X$, $y \in W^{u}(x)$, there is a linear isomorphism $H^{u}_{xy}: \cE_x \to \cE_y$ with the properties:
	\begin{enumerate}
		\item [(a)] For $x \in M$ and $y,z \in W^{u}(x)$ we have $H_{xx} = I_x$ and $H^{u}_{yz} \circ H^{u}_{xy} = H^{u}_{xz}$;
		\item [(b)] For $x \in M$ and $y\in W^{u}(x)$, $t \in \R$:
		\[H^{u}_{xy} = (A^t_y)^{-1} \circ H^{u}_{\varphi^t(x) \varphi^t(y)} \circ A^t_x.\]
        \item [(c)] There is an $r > 0$ and a constant $C > 0$ such that for $x \in M$ and $y \in W^{u}(x)$ with $d(x, y) \leq r$ we have:
\[   \|H^{u}_{xy} - I_{xy}\| \leq C d(x,y)^\alpha.\]
	\end{enumerate}

Moreover, the family of linear maps $H$ satisfying the above properties is unique. The analogous holonomies $H^{s}$ over $W^{s}$ exist, are unique, and satisfy analogous properties.
\end{proposition} 

For $M$ whose sectional curvatures are pinched in $(-4, -1]$ we may conclude by Theorem \ref{regularity} that the horospherical foliations are then $\Conea$ for some $\alpha \in (0,1)$ and moreover by the above we see that $D\varphi^t|_{E^{u,s}}$ is $1$-bunched, so the derivative cocycle restricted to $E^{u,s}$ admits holonomies as above. 

Since for any Anosov flow the foliations $W^u, W^c$ and  $W^s$ locally define a product structure, and by (b) the holonomies are equivariant with respect to the cocycle $D\varphi^t$, the holonomies may be extended equivariantly with respect to $D\varphi^t$ to the connected component of $W^{cu}_{loc}(p):= W^{cu}(p) \cap B$, where $B$ is a  small neighborhood containing $p$. With the additional global product structure of the foliations for the geodesic flow explained in the previous section, by flowing back to the box $B$ for our case the cocycle holonomies may be extended to entire center-unstable leaves. We write $H^{cu, cs}$ for the cocycle holonomies over the leaves of $\wt{W}^{cu, cs}$.

\subsubsection{The Kanai connection} \label{sssec: kanai} Recall that $\wt{V}:= S\wt{M}$. For this section, we assume crucially that the horospherical foliations are $C^1$, or what is equivalent, the splitting $T\wt{V} = E^s \oplus E^u \oplus \R X_g$ is $C^1$. In the study of such geodesic flows, a key tool for the study of rigidity relating the geometry of the horospherical foliations and the dynamics of the geodesic flow is the \textit{the Kanai connection}. 

To motivate its definition, consider the symmetric bilinear tensor on $SM$ known as the Cartan-Kanai form:
$$h_{KC}(v,w) = \omega(v, Iw) + \lambda \otimes \lambda(v,w), $$
where $\lambda$ is as in Subsection \ref{ssec: geodesic}, $\omega$ is a symplectic form on $TM$ given by $\omega = d\lambda$ and $I$ is a $(1,1)$-tensor defined by $I(v_u) = v_u$ for $v_u \in E^u$, $I(v_s) = -v_s$ for $v_s \in E^s$ and $I(X_g) = 0$. The tensor $h_{KC}$ is in fact a pseudo-Riemannian metric  and it is of regularity as high as that of the bundles $E^u$ and $E^s$, which in our case are assumed to be $C^1$. 

Since $h_{KC}$ is at least $C^1$, the pseudo-Riemannian metric defines a connection via Koszul's formula with respect to which it is parallel -- this is the Kanai connection, whose properties we list below. For the following, observe that the Lie bracket of $C^1$ vector fields, and hence the torsion of a $C^0$-connection (c.f. Subsection \ref{ssec: conn}) such as $\nabla$, is well defined in local coordinates.

\begin{proposition} \cite[Theorem 2.6]{mp} \label{kanai}
	There exists a continuous connection $\nabla$, called the Kanai connection, on $T\wt{V}$ with the following properties:
	\begin{enumerate}
		\item [(a)] $h_{KC}$ as defined in above is parallel with respect to $\nabla$, i.e. $\nabla h_{KC} = 0$ and $\nabla$ has torsion $\omega \otimes X_g$;
		\item [(b)] $\nabla$ is $\Gamma$-invariant, $\varphi^t$-invariant, $\nabla \omega = 0$, $\nabla_{X_g} = \mathcal{L}_{X_g}$ and $\nabla X_g = 0$;
		\item [(c)] The Anosov splitting is invariant under $\nabla$, that is, if $X_s \in \Gamma(E^s)$, $X_u \in \Gamma(E^u)$ and $Y$ is any vector field on $\wt{V}$ then 
		$$ \nabla_Y X_s \in \Gamma(E^s), \hspace{1cm} \nabla_Y X_u \in \Gamma(E^u);$$
		\item [(d)] The restriction of $\nabla$ to the leaves of the foliations $\wW^u$ (resp. $\wW^s$) of $S\wt{M}$ is continuously $C^\infty$ along $\wW^u$ (resp. $\wW^s$), in the sense of Subsection \ref{ssec: conn}.  It is moreover flat, complete, and the parallel transport determined by it coincides with the holonomy transport determined by the stable and unstable foliations and with $H$ as in Proposition \ref{holdef} by the uniqueness part of that proposition.
	\end{enumerate}
\end{proposition}
\begin{proof}
	The only non-standard claim not found in \cite{mp} is the completeness claim in item (d) which we prove now.
	
	Since $\nabla$ is  $\Gamma$-invariant, it descends to the quotient $V$; by compactness of $V$, there exists some constant $c > 0$ such that for $Y^-\in E^s$ with $\|Y^-\| < c$ the $\nabla$-geodesic in the direction of $Y^-$ is defined for at least time $1$. But in fact for any $Y^- \in E^s$, there exists a $t > 0$ such that $\|D\varphi^t(Y^-)\| < c$. Then the $\nabla$-geodesic issued from $D\varphi^t(Y^-)$ is well-defined for time at least 1, so by invariance of $\nabla$ by the geodesic flow, so is the $\nabla$-geodesic issued from $Y^-$. These geodesics lift from $V$ to $\wt{V}$.
\end{proof}

\begin{remarks}
	\begin{enumerate} \item [(i)] Since we will not use the Levi-Civita connection of the Sasaki metric on $\wt{V}$, we use $\nabla$ simply to denote the Kanai connection.
	\item [(ii)] The Kanai connection $\nabla$ in fact has Hölder-continuous regularity under $\quart$-pinching, but this will not be used.
	\end{enumerate}
\end{remarks}

Finally, we use all the constructions above to induce a $C^1$-structure on $\partial M$ and then recall some of its properties. 

Consider two leaves of the stable foliation $\wW^s(x)$ and $\wW^s(y)$ lifted to the universal cover $\wt{V}$ and let $\pi^s_y$ and $\pi^s_x$ be defined as in Subsection \ref{ssec: geodesic}. There is a well-defined map $$h^{cu}_{xy}: \wW^{s}(x)\setminus \wW^{cu}(\i(y)) \to \wW^{s}(y)\setminus \wW^{cu}(\i(x))$$ given by $z \mapsto \wW^{cu}(z) \cap \wW^s(y)$, which we call the (center-unstable) \textit{global holonomy} map. Observe that this agrees with $(\pi^s_y)^{-1} \circ \pi^s_x$ on its domain of definition and since $\wW^{cu}$ is $C^1$, $h^{cu}_{xy}$ is $C^1$ so $\pi^s_x$ and $\pi^s_y$ define a $C^1$-structure on $\partial M$ with transition map $(\pi^s_y)^{-1} \circ \pi^s_x = h^{cu}_{xy}$.

\section{Theorem \ref{lie}: Local Compactness of the Foliated Centralizer} \label{sec: compact}

We start with a straightforward proposition mentioned in the introduction proving that the (non-foliated) centralizer of the geodesic flow in the universal cover is \textit{not} a Lie group, and hence too large to be a rigid invariant of hyperbolic metrics.

\begin{proposition} \label{notcomp} Let $(\tilde{M},g)$ be a simply connected manifold with sectional curvatures $K <0$. The set of diffeomorphisms in $\Diff^\infty(S\wt{M})$ commuting with the geodesic flow is infinite dimensional in any of the weak $C^k$-topologies.
\end{proposition}
\begin{proof}
	
	Fix some $v_0 \in S\wt{M}$. Let $S_0$ be the $C^\infty$ submanifold of codimension 1 in $S\wt{M}$ obtained by restricting the bundle $S\wt{M}$ to the horosphere $B_{(v_0)_+}(p_0) \subseteq \wt{M}$, where $p_0$ is the basepoint of $v_0$. Let $S \subseteq S_0$ be a small precompact neighborhood of $v_0$ in $S_0$. Observe that $S$ has codimension $1$ in $S\wt{M}$ and is transverse to $X_g$, the geodesic spray. 
	
	Let $\Phi: S \times \R \to S\wt{M}$ be given by $\Phi(v, t) = \varphi^t(v)$. Then as long as $S$ is chosen sufficiently small, $\Phi$ is a diffeomorphism onto its image since for $v \in S$ its trajectory $\varphi^t(v)$ will not intersect $S \subseteq B_{(v_0)_+}(p_0)$ again. Moreover, in $\Phi$-coordinates, $X_g$ is given simply by $0 \times \partial/\partial t$. 
	
	Hence, any vector field $Z$ on $\Phi(S \times \R)$ of the form $Z = Y \times 0$ (in $\Phi$-coordinates), where $Y$ is a compactly supported vector field on $S$, will commute with $X_g = 0 \times \partial/\partial t$ and can be extended to all of $S\wt{M}$. The Lie algebra of such vector fields is evidently infinite-dimensional, so the centralizer of $X_g$ on $S\wt{M}$ is not finite-dimensional in any of the weak $C^k$-topologies.
\end{proof}

 Now we move to the proof of Theorem \ref{lie}, regarding the finite dimensionality of the group $G$ under the assumptions of Theorem \ref{main}.

\begin{proof}[Proof of Theorem \ref{lie}]

We start by constructing normal coordinates on the unstable/stable leaves with respect to which the elements of $G$ have a simple (affine) representation.

\begin{proposition} \label{normalforms} There exists a family of $C^\infty$ diffeomorphisms $\{\cH^u_x\}_{x \in \wt{V}}$, with $\cH^u_x: T_x \wW^u(x) \to \wW^u(x)$ such that:
	\begin{enumerate}
		\item [(a)] $\cH^u_x(0) = x$ and $D_x\cH^u_x$ is the identity map for each $x  \in \wt{V}$;
		\item [(b)] The maps $\cH_x^u$ depend continuously $C^\infty$ on $x$ along $W^u$ and smoothly on $x$ restricted to a leaf of $W^u$;
		\item [(c)] For $\phi \in G$ and any $x \in \wt{V}$, we have:
		$$\phi|_{\wW^u(x)} = \cH^u_{\phi(x)} \circ (D_x\phi)|_{E^u(x)} \circ (\cH_x^u)^{-1}. $$
	\end{enumerate}
	The analogous propositions hold for the foliations $\wW^{cu}$, $\wW^s$, and we denote the coordinates, respectively as $\cH^{cu}$ and $\cH^s$.
\end{proposition}
\begin{proof}
	
	By definition any element $\phi \in G$ preserves the bundles $E^u$, $E^s$ and since $\ker{\lambda} = E^u \oplus E^s$ as well as (by definition) $d\phi(X_g) = X_g$ it follows that $\phi_* \lambda = \lambda$.  Thus since $\phi$ preserves $h_{KC}, \omega$ and $X_g$, the Kanai connection, being the unique connection with respect to which $h_{KC}$ is parallel with torsion $\omega \otimes X_g$, is invariant by the action of $\phi \in G$. 
	
	Recall that by Proposition \ref{kanai} (a), (d), the connection $\nabla^u_x$ obtained by restricting the Kanai connection to a single unstable leaf $\wW^u(x)$ is smooth, flat and torsion-free, so in particular it is the Levi-Civita connection of a smooth flat Riemannian metric on $\wW^u(x)$. By  Proposition \ref{kanai} (d) $\nabla^u_x$ is complete and thus so is the Riemannian metric. Then let $\cH^u_x: T_x \wW^u(x) \to \wW^u(x)$ be the exponential map of $\nabla^u_x$, which is a diffeomorphism $\nabla^u_x$ is the connection of a flat Riemannian metric. Moreover, $\cH^u_x$ is smooth on leaves of $\wW^u$ and continuously $C^\infty$ along leaves of $\wW^u$ by Proposition \ref{kanai} (d) again. Finally, for (c), we observe again that elements of $G$ preserve the Kanai connection and so $\phi_*\nabla^u_x = \nabla^u_{\phi(x)}$ from which the statement follows.
	
	The proof for the stable foliation follows using the time-reversing isomorphism $\i$ of $\wt{V}$, and the proof for the weak unstable foliation by equivariance of all the constructions above with respect to the geodesic flow $\varphi^t$. Precisely, one defines $\cH^{cu}$ in the unique $\varphi^t$-equivariant way $\cH^{cu}_x(y) = \varphi^t \circ \cH^u_x \circ \varphi^{-t}(y)$ where $t \in \R$ is such that $\varphi^{-t}(y) \in \wW^u(x)$. Then (a) and (b) are immediate, whereas (c) follows since elements $\phi \in G$ by definition commute with $\varphi^t$.
\end{proof}
\begin{remark}
	The construction above is similar to a well-known result first proved by Guysinsky and Katok in \cite{gk} constructing normal forms for uniformly contracted bundles. In the non-compact setting, this result does not apply so we use the geometric structure provided by $h_{KC}$, an idea which has also been introduced before in many works. 
\end{remark}

For $\psi \in G$, pick some arbitrary $p \in \wt{V}$, a compact neighborhood $K$ of $p$ and a bounded open set $U$ such that $\psi(K) \subseteq U$. Then $U_p := \{\phi \in G: \phi(K) \subseteq U \}$ is a neighborhood of $\psi$ in the compact open topology. We will show that the closure of $U_p$ is compact, which proves Theorem \ref{lie} by Theorem \ref{mz}.
	
The compact-open topology is second countable \cite[p. 35]{hir}, so it suffices to check sequential compactness of $\overline{U}_p$. Let $\phi_n$ be a sequence in $U_p$ and observe that by passing to a subsequence we may assume that $q_n : = \phi_n(p)$ converges to some $q \in \overline{U}$. Moreover, we claim that we may pass to a further subsequence so that $D_p\phi_n$ also converges to some linear isomorphism $A: T_p \wt{V} \to T_q \wt{V}$. 

To prove the claim, it suffices to show that there is some $C > 1$ such that $ C^{-1} \|v\| \leq \|D_p \phi_n (v) \| \leq C\|v\|$ for all $n$ and $v \in T_p \wt{V}$. This is evidently true for $v = X_g$. Further, if $ \|D_p \phi_n (v_n) \|/\|v_n\| \to \infty$ and $v_n \in E^u$,  by Proposition \ref{normalforms}, the normal form charts are continuously $C^\infty$ along $W^u$, this would contradict $\phi_n(K) \subseteq U$, since $U$ is a bounded set. The same holds for $v_n \in E^s$. Finally, if $ \|D_p \phi_n (v_n) \|/\|v_n\| \to 0$ and $v_n \in E^u$, since the $\phi_n$ preserve $\omega$ and $E^s$ and $E^u$ are transverse Lagrangian subspaces there must exist some $u_n \in E^s$ such that $ \|D_p \phi_n (u_n) \|/\|u_n\| \to \infty$, reducing to the previous case. Similarly there cannot exist $ \|D_p \phi_n (v_n) \|/\|v_n\| \to 0$ and $v_n \in E^s$, so the claim is proved.

Using $A$ and the normal form coordinates $\cH_x$ defined above we now construct a map $\phi_0 \in G$ which we show is the $C^0$-limit of $\{\phi_n\}_{n \in \N}$. Let $\psi^s: \wW^s(p) \to \wW^{s}(q)$ and $\psi^{cu}: \wW^{cu}(p) \to \wW^{cu}(q)$ be given by:
$ \psi^{\bullet} := \cH^\bullet_{\phi(x)} \circ A|_{E^\bullet(x)} \circ  (\cH_x^\bullet)^{-1}, \text{ where } \bullet \in \{cu, s\},$
so that  the maps $\phi_n|_{W^\bullet(p)}$ converge uniformly $C^\infty$ to  $\psi^\bullet$. To extend the above construction to an open set of $\wt{V}$, we use the product structure of the foliations on the universal cover. For $x \in \wt{V}$ there exists a $C^1$ diffeomorphism $[\cdot, \cdot]_x: \wW^{cu}(x) \times \wW^s(x) \to \wt{V} \setminus (\wW^{cu}(\i(x)) \cup \wW^s(\i(x)))$ called the \textit{Bowen bracket} given by $(y,z) \mapsto [y,z]_x := \wW^s(y) \cap \wW^{cu}(z).$
Define the diffeomorphism $\phi_0: \wt{V} \setminus (\wW^{cu}(\i(p)) \cup \wW^s(\i(p))) \to \wt{V} \setminus (\wW^{cu}(\i(q)) \cup \wW^s(\i(q)))$ by:
$$ \phi_0([y,z]_p) = [\psi^{cu}(y), \psi^{s}(z)]_q,$$
where in the equation above we use the fact that the Bowen bracket is a diffeomorphism  $\wW^{cu}(x) \times \wW^s(x) \to \wt{V} \setminus (\wW^{cu}(\i(x)) \cup \wW^s(\i(x)))$. Furthermore, since the foliations are $C^1$ the dependence of the bracket $[\cdot, \cdot]_x$ on $x$ is also $C^1$. Using the facts just mentioned along with the fact that elements of $G$ preseve the bracket (since they preserve the foliations $\wW^{cu,s}$) we see that $\phi_0$ agrees with the limit of $\phi_n$, that is: as maps on $(y,z) \in \wW^{cu}(p) \times \wW^s(p) \cong \wt{V} \setminus (\wW^{cu}(\i(p)) \cup \wW^s(\i(p)))$, we have the following uniform $C^1$-convergence on compact sets
$$[\phi_n(y), \phi_n(z)]_{q_n} \to [\psi^{cu}(y), \psi^s(z)]_q = \phi_0([y,z]_p).$$

Now we verify that $\phi_0$ is uniformly $C^\infty$ along $W^{cu}$-leaves on the open set $\wt{V} \setminus (\wW^{cu}(\i(p)) \cup \wW^s(\i(p)))$. Fix some $x \in \wt{V} \setminus (\wW^{cu}(\i(p)) \cup \wW^s(\i(p)))$. Then by the convergence in the previous paragraph, we have that $(\phi_n)|_{W^{cu}(x)} \to (\phi_0)|_{W^{cu}(x)}$ uniformly on compact sets. On the other hand, since $\phi_n \in G$, by Lemma \ref{normalforms}  the restrictions $(\phi_n)|_{W^{cu}(x)}$ are given by:
 	$$(\phi_n)|_{W^{cu}(x)} = \cH^{cu}_{\phi_n(x)} \circ (D_x\phi_n)|_{E^u(x)} \circ (\cH_x^{cu})^{-1}, $$
and hence their limit is given by $\cH^{cu}_{\phi_0(x)} \circ A' \circ (\cH_x^{cu})^{-1},$ where $A'$ is the limit of $(D_x\phi_n)|_{E^u(x)}$. Therefore $(\phi_0)|_{W^{cu}(x)} = \cH^{cu}_{\phi_0(x)} \circ A' \circ (\cH_x^{cu})^{-1}$, which is $C^\infty$, and uniformly so along leaves since the right-hand side expression clearly is. By the same argument, $\phi_0$ is uniformly $C^\infty$ along $W^s$-leaves. Hence we conclude by Journé's Lemma (Lemma \ref{journe}) that $\phi_0$ is $C^\infty$ on $\wt{V} \setminus (\wW^{cu}(\i(p)) \cup \wW^s(\i(p)))$.

Finally, it remains to extend $\phi_0$ to all of $\wt{V}$. To do this, simply repeat the same entire argument with $\i(p)$ in place of $p$ and then passing to a further subsequence finally constructs a $\phi$ such that $\phi_n \to \phi$ in the compact-open topology. 
\end{proof}

\section{Smoothness of Foliations: Proof of Theorem \ref{main}} \label{sec:smooth}

 From now on, we assume that $\wt{V}$ is the universal cover of the unit tangent bundle of a manifold satisfying the hypotheses of Theorem \ref{main}.
 

By Theorem \ref{lie}, $G$ is a Lie group and we let its Lie algebra be denoted by $\mathfrak{g}$. We identify $\lieg$ with the space of vector fields generating the action of $G_0$, the connected component of the identity of $G$. With this identification, for $x \in \wt{V}$ we write $\lieg_x\subseteq T_x \wt{V}$ for the subspace spanned by the vector fields in $\mathfrak{g}$ evaluated at $x$. The orbits $G_0 x$ are immersed submanifolds partitioning $\wt{V}$ which are tangent to $\lieg_x$. A key observation is that while the vector fields in $\g$ themselves are not invariant by the action of $\Gamma$ on $\wt{V}$, the subspaces $\g_x$ are.  

Observe that $x \mapsto \dim \lieg_x$ is a lower semi-continuous function, since $$\dim \lieg_x = \dim G - \dim \ker (D_{id}\theta_x),$$ where $\theta_x(\psi) := \psi(x)$ for $\psi \in G$ and $x \in S\wt{M}$, and $(\psi,x) \mapsto \theta_x(\psi)$ is $C^\infty$ so that $ \dim \ker (D_{id}\theta_x)$ is upper semi-continuous in $x$. 

Let $d_0 = \max_{x \in \wt{V}} \dim \lieg_x$. Then by lower semi-continuity $$\cU = \{x \in \wt{V}: \dim \lieg_{x} = d_0\}$$ is an open, non-empty subset of $\wt{V}$. Let $\g^u$, $\g^c$ and $\g^s$ be the images of $\lieg$ under the projections to $E^u$, $\R X_g$ and $E^s$, respectively and again observe that $x \mapsto \dim \lieg^{u,s}_x$ are lower semi-continuous functions on $\cU$. Let $d_{u, s} := \max_{x \in \cU} \dim \lieg_{x}^{u,s}$, that is, the maximum transverse dimension of an orbit of a leaf of $W^{cu, cs}$ under the action of $G_0$. By lower semi-continuity again, the sets $$\cU^{u,s} = \{x \in \cU: \dim \lieg_{x}^{u,s} = d_{u,s}\}$$ are open and by definition non-empty. Observe that $\cU^{u,s}$ are both $\Gamma$ and $\varphi^t$-invariant so that they are dense on $\wt{V}$ by topological transitivity of $\varphi^t$. 

\begin{lemma} \label{notcenter}
	Either $d_u$ or $d_{s}$ is  positive. 
\end{lemma}

\begin{proof} Suppose otherwise for contradiction. This implies that $\mathfrak{g}_x \subseteq \R X_g$ for all $x \in \wt{V}$, so any vector field in $\mathfrak{g}$ must fix the orbits of the geodesic flow on $\wt{V}$. Hence any vector field in $\mathfrak{g}$ must be of the form $fX_g$, where $f \in C^\infty(\wt{V})$ and $X_g$ is the geodesic spray. But notice that the vector field $f X_g \in \mathfrak{g}$ must preserve the contact form $\lambda$ on $\wt{V}$ since it preserves $X_g$ and $E^u \oplus E^s$. On the other hand:
$$\mathcal{L}_{fX_g} \lambda = d \circ i_{fX_g}\lambda + i_{X_g} d\lambda = df, $$
since $i_{X_g} d\lambda = 0$ and $\lambda(X_g) = 1$. Hence $\mathcal{L}_{fX_g} \lambda$  is zero if and only if $f$ is constant. Therefore, if $d_u$ and $d_s$ are both equal to $0$ the action of $G_0$ reduces to the action of the geodesic flow, which contradicts the hypothesis that the reduced centralizer $G_c = G/\langle \varphi^t \rangle$ is not discrete. 
\end{proof}

Suppose without loss of generality that $d_s > 0$. The first key step in the proof is to show that in fact $d_s = d$, which we prove in the following proposition. The idea is to to identify $\partial \wt{M}$ with the space of center unstable leaves and construct a $\Gamma$-invariant foliation of $\partial \wt{M}$ whose leaves are given by $G_0$-orbits of center unstable, i.e, $W^{cu}$-leaves. Then a classic argument using the north-south dynamics of the action of $\Gamma$ on $\partial \wt{M}$ in fact shows that this foliation must be trivial, that is, $d_s = d$. Since the action of $G_0$ is $C^\infty$ and commutes with the geodesic flow, this yields that $\wt{W}^{cu}$ itself is a $C^\infty$ foliation on an open and dense set of $V$.

\begin{proposition} \label{dsismax}
	For $d_s$ as above, $d_s = d$. In particular, on the set $\cU^s$ the action of $G_0$ on the space of unstable leaves is locally homogeneous and the foliation $\wt{W}^{cu}$ is $C^\infty$ on  $\cU^s$.
\end{proposition}
\begin{proof}
 Fix some $x_0 \in \cU^s$ such that its $\varphi^t$-orbit in $V = \wt{V} / \Gamma$ is periodic. Since $\wW^{cs}(x_0)$ is contracted under $\varphi^t$, this implies that $\wW^{cs}(x_0)\subseteq \cU^s$ by $\Gamma$ and $\varphi^t$-invariance of $\cU^s$. Let $x_1 \in \cU^s$ be another $\varphi^t$-periodic orbit which moreover satisfies $\wW^{cs}(x_1) \neq \wW^{cs}(x_0)$, which exists since $\cU^s$ is open and non-empty. Thus $\g^s$ has constant dimension on $\wW^{s}(x_i)$, $i = 0,1$. Moreover, we have:
 
 \begin{claim}$\g^s$ is involutive, i.e. $[\g^s, \g^s] \subseteq \g^s$.
 	\end{claim} 
 \begin{proof} For a vector field $X$ on $\wt{V}$ we write $X = X^u + X^s + X^c$ for the decomposition corresponding to $T\wt{V} = E^u \oplus E^s \oplus E^c$, where $E^c = \R X_g$. Similarly, for a set $\h$ of vector fields on $\wt{V}$ we write $\h^u, \h^s$ and $\h^c$ the sets of vector fields obtained via this decomposition.  With this notation, to show that $\g^s$ is involutive it suffices to show that for $X,Y \in \g$ we have $[X^s, Y^s] \in \g^s$. For $X, Y \in \g$, expanding the Lie bracket with the decompositions for $X$ and $Y$ we have:
 	$$[X, Y]^s = [X^s, Y^s]^s + [X^c, Y]^s + [X, Y^c]^s $$
as the other Lie brackets have no $E^s$-component, since it is a standard fact for contact Anosov flows that $[E^s, E^u]$ is always parallel to $E^c$, and obviously we have $[E^u, E^u] \subseteq E^u$. Since $\g$ is tangent to orbits of $G_0$, it must be involutive so we have $[\g, \g]^s \subseteq \g^s$, and thus $[X,Y]^s \in \g^s$. Since clearly also   $ [X^s, Y^s] = [X^s, Y^s]^s$, to show that $ [X^s, Y^s] \in \g^s$ it suffices to show that  $[X^c, Y]^s$ and  $[X, Y^c]^s$ are also in $\g^s$.
 
We check that $[X^c, Y]^s \in \g^s$, the proof for  $[X, Y^c]^s$ being analogous. Since $E^c  = \R X_g$, we may write $X^c = f X_g$ for some function $f \in C^\infty(\wt{V})$. But now recall that by definition since $Y \in \g$, it commutes with $X_g$ i.e. $[Y, X_g] = 0$, so:
$$[X^c, Y]  = [fX_g, Y] = f(X_gY-YX_g) - Y(f) X_g = -Y(f)X_g,$$
so that we see in fact that $[X^c, Y]^s = 0$.

 \end{proof}

Hence $\g^s$ integrates to foliations $\cF_i$, $i =0,1$ of $\wW^{s}(x_i)$, $i = 0,1$. Let $h^{cu}: \wW^{s}(x_0)\setminus \wW^{cu}(\i(x_1)) \to \wW^{s}(x_1)\setminus \wW^{cu}(\i(x_0)) $ be the center-unstable holonomy map namely, $h^{cu}(z) = \wW^{cu}(z) \cap \wW^s(x_1)$ for $z \in \wW^{s}(x_0)\setminus \wW^{cu}(x_1)$. The following lemma which uses crucially the pinching hypothesis on the metric shows that $\cD := \g^s$ is invariant under the holonomy maps. The idea is that high regularity subbundles of pinched flows must in fact be holonomy invariant, but extra care has to be taken in a non-compact setting.

\begin{lemma} \label{holinv}The distribution $\cD$ over $\wW^s(x_0)$ is mapped to the distribution $\cD$ over $\wW^s(x_1)$ by the holonomies $h^{cu}$. Namely, for all $y \in \wW^s(x_0)\setminus \wW^{cu}(x_1)$,  $Dh^{cu}(\cD_y) = \cD_{h^{cu}(y)}$.
\end{lemma}
\begin{proof}

	Denote by $H^{cu}$ the cocycle holonomies of the bundle $E^s$ over leaves of $\wW^{cu}$ as given by Proposition \ref{holdef}, and extended to the center-unstable leaves by the remark following the Proposition. Note that $Dh^{cu}$ defines a family isomorphisms of $E^s$ over $\wW^{cu}$ which satisfies the conditions of Proposition \ref{holdef} so by uniqueness in the proposition we see that $Dh^{cu} = H^{cu}$, in their common domain of definition. We will now show invariance of $\cD$ under $H^{cu}$ between the two leaves in question.
	
	To deal with the lack of uniformity of $\cD$ in our non-compact setting we have to begin by verifying invariance for points negatively assymptotic to a periodic orbit in the quotient. Namely, let $y \in \wt{W}^{s}(x_0)$ be such that its negative endpoint $y^{-} \in \partial M$ is the endpoint of some $\varphi^t$-periodic orbit $w \in \wt{V}$, i.e., such that $w_- = y_-$ and $w$ is periodic in $V$ and we moreover assume that $w$ also lies in $\cU^s$, which is possible since $\cU^s$ is open and dense. Let $z = \wW^{cu}(y) \cap \wW^{s}(x_1)$, i.e. $z = h^{cu}_{x_0x_1}(y)$. 
\begin{claim} \label{holinvper} For $y,z$ as above $H^{cu}_{y,z} \cD_y = \cD_z$.
\end{claim}
\begin{proof}
	
		Recall that $d_0$ is the dimension of the distribution $\cD$ on $\cU^s$. Then $\cD$ defines a section of the Grassmanian bundle $\cP \to \cU^s$ with fibers $\cP_x$ isomorphic to Gr$(d_0, d)$ of $d_0$-dimensional subspaces of $E^s$. As is standard, we endow the fibers $\cP_x$ with the metric:
	$$\rho_x(V, W) = \sup_{v \in V, \|v\| = 1} \inf_{w \in W} \|v - w\|$$ 
	For this metric, given a linear isomorphism $A$ acting on the underlying vector space, the induced action of $A$ on the Grassmannian manifold is bi-Lipschitz with norm bounded by $\|A\|\|A^{-1}\|$. The action of $D\varphi^t$ on $T\wt{V}$ induces a natural action on $\cP$. By the pinching hypothesis, we obtain the following bound for $V, W \in \cP_x$:
	$$\begin{aligned}
		\rho_x(V, W) &\leq \|D_x\varphi^{-t}|_{E^s}\| \|D_{\varphi^{-t}(x)}\varphi^{t}|_{E^s}\| \, \rho_{\varphi^{-t}x}(D\varphi^{-t}V, D\varphi^{-t} W) \\
		& \leq C e^{-t} e^{\tau t}\, \rho_{\varphi^tx}(D\varphi^{-t}V, D\varphi^{-t} W)\\
		& \leq C e^{(\tau-1)t}\rho_{\varphi^tx}(D\varphi^{-t}V, D\varphi^{-t} W)
	\end{aligned}$$
	 for all $t >0 $, where $C$ is some constant which does not depend on any of the parameters and where $\tau = \sqrt{a}$ and $a < 4$ is such that the sectional curvatures of $M$ are pinched in $[-a, -1]$. 
	
	For $t \in \R$, we write $y_t = \varphi^t(y)$ and $z_t = \varphi^t(z)$. Let $t_0 \in \R$ be such that $z_{t_0} \in \wt{W}^u(y)$. Recall that by Proposition \ref{holdef} (c), $H^{cu}$ is given by the limit $$H^{cu}_{yz} = \lim_{t \to \infty} D \varphi^{t-t_0} \circ I_{y_{-t} z_{-t+t_0}} \circ D\varphi^{-t},$$
	
 Therefore, to conclude the proof it suffices to show that 
	$$  \lim_{t \to \infty} \rho_z(D \varphi^{t-t_0} \circ I_{y_{-t} z_{-t+t_0}} \circ D\varphi^{-t}(\cD_y), \cD_z) = 0.$$
	
	 Using $D\varphi^t$ invariance of $\cD$, for $t > t_0$:
$$\begin{aligned}&\rho_z(D \varphi^{t-t_0} \circ I_{y_{-t} z_{-t+t_0}} \circ D\varphi^{-t}(\cD_y), \cD_z) \\ 
	&\leq  Ce^{(\tau -1)t} \rho_z( I_{y_{-t} z_{-t+t_0}} \circ D\varphi^{-t}(\cD_y), D\varphi^{-t+t_0}\cD_z) \\
	&= Ce^{(\tau -1)t}  \rho_z( I_{y_{-t} z_{-t+t_0}} \cD_{y_{-t}}, \cD_{z_{-t+t_0}}).
	\end{aligned} $$

It remains to bound $ \rho_z( I_{y_{-t} z_{-t+t_0}} \cD_{y_{-t}}, \cD_{z_{-t+t_0}})$. From the construction of the map $I_{y_{-t} z_{-t+t_0}}$ (Proposition \ref{holdef} (c)), the family of maps $I$ has $C^1$ dependence on the basepoints since the bundle $E^s$ itself is $C^1$. Moreover, recall that the distribution $\cD$ is also of class $C^1$ on $\cU^s$, in fact $C^\infty$, but it may fail to be uniformly so since $\cU^s$ is an open set. However, since both $y_{-t}$ and $z_{-t+t_0}$ are negatively asymptotic to $p_-$, i.e., they shadow a periodic point $p \in \cU^s$ as $t \to \infty$, for sufficiently large $t$ they must lie within some compact set $N \subseteq \cU^s$ which is a closed tubular neighborhood of the periodic point $p$. Therefore $\cD$ is uniformly $C^1$ on the orbits $y_{-t}$ and $z_{-t+t_0}$ for $t$ sufficiently large so that  that $\rho_z( I_{y_{-t} z_{-t+t_0}} \cD_{y_{-t}}, \cD_{z_{-t+t_0}}) \leq C^{''} d(y_{-t}, z_{-t+t_0})$, since both the distribution $\cD$ and the maps $I$ are uniformly $C^1$ on the basepoints and hence uniformly Lipschitz. The distance $d(y_{-t}, z_{-t+t_0}) \leq C''' e^{-t}$, since $y \in \wW^{cu}(z)$ with a rate of $e^{-t}$, which suffices to conclude the proof, since we have:
$$\begin{aligned}&d_z(D \varphi^{t-t_0} \circ I_{y_{-t} z_{-t+t_0}} \circ D\varphi^{-t}(\cD_y), \cD_z) \\
	&\leq Ce^{(\tau -1)t}  d_z( I_{y_{-t} z_{-t+t_0}} \cD_{y_{-t}}, \cD_{z_{-t+t_0}}) \\
	& \leq Ce^{(\tau -1)t} C^{''} d(y_{-t}, z_{-t+t_0}) \\
	&\leq C e^{(\tau -1)t} e^{-t}  = C e^{(\tau -2)t}\to 0,
\end{aligned} $$
where we absorb all the constants in the last line into one.
\end{proof}

With the claim above, we now conclude the proof of Lemma \ref{holinv}. For this, it suffices to observe that since $\cU^s$ is open dense and the periodic orbits of an Anosov flow are also dense, there exists a dense set in $\wW^{s}(x_0)$ of $y$ satisfying the hypotheses of Lemma \ref{holinvper}. That is, for $y$ on a dense set of $\wW^{s}(x_0)$ we have $Dh^{cu}_{yz} \cD_y = \cD_z$. But the maps $Dh^{cu}_{yz}$ depend $C^0$ on the basepoint $y$ and so does the distribution $\cD$, so that  by continuity $Dh^{cu}_{yz} \cD_y = \cD_z$ holds for all $y \in \wW^{s}(x_0) \setminus \wW^{cu}(x_1)$, as desired.
\end{proof}

We now finally conclude the proof of Proposition \ref{dsismax}. Since $\wW^{cs}(x_1) \neq \wW^{cs}(x_0)$, the natural projections $\wW^{s}(x_i) \to \partial M$, $i = 0,1$, form a system of charts with transition maps $h^{cu}$, which by Subsection \ref{ssec:hols} are $C^1$, and thus give $\partial M$ a $C^1$ structure. Moreover, since $h^{cu}(\cF_0) = \cF_1$ there is a well-defined  topological foliation $\cF$ of $\partial M$ associated to the distribution $\cD$ (c.f. Subsection \ref{sssec: kanai}). Since the distribution $\cD_p$ is  $\Gamma$-invariant over the $\Gamma$-orbits of the $x_i$, we can consider the leaves of $\wW^s$ over $\Gamma \cdot x_i$ ($i = 0,1$) also as charts on $\partial M$, so that there is a well-defined action of $\Gamma$ on $\partial M$ which also preserves the foliation $\cF$. By a well-known argument of Hamenstädt (see \cite[Sec. 4]{ham}), any $\Gamma$-invariant topological foliation of $\partial M$ is trivial, and since it has positive dimension we have $d_u =d$ and  $\cD = E^s$ . 

Hence for $x$ in the open and dense $\cU^s$ we have that $G \cdot \wW^{cu}(x)$ contains an open set in $\wt{V}$ around $\wW^{cu}(x)$; since the action of $G$ preserves the $\wW^{cu}$ foliation and acts by $C^\infty$ diffeomorphisms, this implies that at $x$ the foliation $\wW^{cu}$ is $C^\infty$. To see this consider the map $a_g: G \to \wt{V}$ given by $g \mapsto g \cdot x$; let $D_0$ be a small open disk transverse (and of complementary dimension) to a hypersurface tangent to the kernel of $Da_g$; then $D_0 \times \wW^{cu}(x) \to \wt{V}$ given by $(g, y) \mapsto g \cdot y$ defines a $C^\infty$-diffeomorphism onto its image with the product structure $D_0 \times \wW^{cu}(x)$ compatible with the foliation, namely, $\wW^{cu}$ is a $C^\infty$ foliation in a neighborhood of $\wW^{cu}(x)$.
\end{proof}

The map $\wt{V} \to \partial \wt{M}$ given by $v \mapsto v_-$ induces a homeomorphism of the space $\wt{V}/\wW^{cu}$ of $\wW^{cu}$-leaves to $\partial \wt{M}$, and we identify these from now on. The previous paragraph shows that there exists an open set $U \subseteq \partial \wt{M}$ on which $\wW^{cu}$ is smooth. The subset of $\partial \wt{M}$ on which is $\wW^{cu}$ smooth is invariant by the action of $\Gamma$ on $\partial \wt{M}$ and we use this now to conclude the proof. Each $\gamma \in \Gamma$ acts on $\partial \wt{M} \approx S^n$ with north-south dynamics. Take $\gamma \in \Gamma$ such that $\gamma_-$, the repelling fixed point of $\gamma$ on $\partial \wt{M}$, lies in $U$. This is possible since the set of $(\gamma_-, \gamma_+) \in \partial^2 \wt{M}$ which are endpoints of axes of $\gamma \in \Gamma$ is dense, i.e. the closed geodesics are periodic in the space of geodesics. Then $\gamma^n(U)$ converges to $\partial \wt{M} - \{\gamma_+\}$, where $\gamma_+$, the repelling fixed point of $\gamma$ on $\partial \wt{M}$. Hence we showed that $\wW^{cu}$ is $C^\infty$ except perhaps at the leaf corresponding to $\gamma_+$. To conclude now repeat the argument with any $\gamma'$ such that $\gamma'_+ \neq \gamma_+$. 

Therefore we find that the foliation $\wW^{cu}$ is a $C^\infty$ foliation everywhere. Since $\iota$ maps $\wW^{cu}$-leaves to $\wW^{cs}$-leaves, we similarly conclude that $\wW^{cs}$ is $C^\infty$. The result now follows from the main theorem of \cite{bfl}, which shows that $C^\infty$ smothness of horospherical foliations implies that the geodesic flow of $(M, g)$ is $C^\infty$ conjugate to that of a  manifold with constant negative sectional curvature. Then the minimal entropy rigidity theorem \cite{bcg} shows that $M$ must have constant negative sectional curvature.

\end{document}